\documentclass[11pt,reqno]{amsart}

\usepackage{amssymb,latexsym}
\usepackage{graphicx}
\usepackage{fancyhdr}
\usepackage{url}  

\usepackage{psfig,epsfig,latexsym}
\usepackage{amsmath}
		
\theoremstyle{plain}
\numberwithin{equation}{section}
\newtheorem{thm}{Theorem}[section]
\newtheorem{theorem}[thm]{Theorem}

\newtheorem{conj}[thm]{Conjecture}
\DeclareMathOperator{\Reverse}{Reverse}

\newcommand{\eqn}[1]{(\ref{#1})}

\newcommand{\beql}[1]{\begin{equation}\label{#1}}
\newcommand{\eeq}{\end{equation}}

\newcommand{\sC}{{\mathcal C}}

\newcommand{\sG}{{\mathcal G}}
\newcommand{\sP}{{\mathcal P}}
\newcommand{\sQ}{{\mathcal Q}}

\begin{document}


\setcounter{page}{1}


\title[2178 And All That]{2178 And All That}
\author{N. J. A. Sloane}
\address{The OEIS Foundation\\
                11 South Adelaide Avenue\\
                Highland Park, NJ 08904, USA}
\email{njasloane@gmail.com}

\begin{abstract}
For integers $g \ge 3$, $k \ge 2$, call a number $N$ a
$(g,k)$-\textit{reverse multiple} if the reversal of $N$ in
base $g$ is equal to $k$ times $N$.
The numbers $1089$ and $2178$ are the two smallest $(10,k)$-reverse multiples,
their reversals being $9801 = 9 \cdot 1089$ and $8712 = 4 \cdot 2178$.
In 1992, A. L. Young introduced certain trees in order to study the
problem of finding all $(g,k)$-reverse multiples.
By using modified versions of her trees, which we call \textit{Young graphs},
we determine the possible values of $k$ for bases $g = 2$ through $100$, 
and then show how to apply the transfer-matrix method to enumerate the $(g,k)$-reverse
multiples with a given number of base-$g$ digits.  
These Young graphs are interesting finite
directed graphs, whose structure is not at all well understood.
\end{abstract}

\maketitle

\section{Introduction}\label{Sec1}
For integers $(g,k)$ with $g \ge 3$, $2 \le k <g$, 
call a number $N$ a ($g,k$)-\textit{reverse multiple}
if the reversal of the base-$g$ expansion of $N$ is the base-$g$ expansion of 
the product $kN$. 
The decimal numbers 1089 and 2178 are the two smallest
(10,$k$)-reverse multiples, their reversals being respectively 
$9801 = 9 \cdot 1089$ and $8712 = 4 \cdot 2178$.
There are no other 4-digit examples in base 10.   
In 1940, G. H. Hardy \cite{Hardy} famously remarked that the existence of 
these two numbers was ``likely to amuse amateurs'',
but was not of interest to mathematicians, since this result is
``not capable of any significant generalization''.

It seems fair to say that Hardy was wrong, since references
\cite{GrBa75},
\cite{IDM},
\cite{Kacz68},
\cite{KlSm69},
\cite{Pud07},
\cite{Sut66},
\cite{WeWi13},
\cite{Young1}, \cite{Young2}
discuss generalizations.
References
\cite{Ball},
\cite{Dorrie},
\cite{MMS78},
\cite{Wells}
also mention the problem.
The bibliography lists all the articles or books known to
the author that discuss this topic.  
There may well be other references, since---partly no doubt
because of Hardy's comment---\textit{Mathematical Reviews}
does not cover this subject, and the author
would appreciate hearing about them.  A more appropriate title
for the present paper would have been ``1089 and all that'',
but this was already in use \cite{1089}, prompted
by another interesting property of 1089,
namely that if one takes any three-digit decimal number $abc$ with $a > c+1$,
$c \ne 0$, and performs the successive operations of reverse, subtract,
reverse, add, the result is always 1089.
For generalizations of this property, see \cite{Web95}.

Hardy's book was published in 1940, but apparently it was not
until 1966 that the problem was taken up again, by Alan Sutcliffe \cite{Sut66}.
His paper and subsequent papers by 
Kaczynski \cite{Kacz68},
Klosinski and Smolarski \cite{KlSm69},
Grimm and Ballew \cite{GrBa75}, and
Pudwell \cite{Pud07}
concentrate on finding all $(g,k)$-reverse multiples 
in arbitrary bases $g$ with a specified (and small)
number of digits.
These papers demonstrate that this is a fairly difficult problem,
which even for two- or three-digit numbers
is still not completely solved (see, for example, the table
on page 286 of \cite{Sut66}).

In 1992, Young \cite{Young1}, \cite{Young2}
introduced certain trees in order to study the problem
of finding all $(g,k)$-reverse multiples for a fixed base $g$.
The present paper extends her work.
We replace her trees with certain finite directed graphs that we refer to as 
``Young graphs''. (`Young diagram'' would
have been a better name, but that term is already in use.)
Once one has the Young graph for a particular pair $(g,k)$, 
it is easy to generate as many examples
of $(g,k)$-reverse multiples as one wishes. 
It is also easy to program a computer to determine
whether the Young graph exists, and hence to find,
for any given value of $g$, the possible values of the multiplier
$k$ (see Table \ref{gkTable} in \S\ref{Sec4} for bases $3 \le g \le 20$).
Furthermore, by applying the transfer-matrix method 
from combinatorics \cite[\S4.7]{Stanley}
to the Young graph, one can obtain a generating function
for the number of $(g,k)$-reverse multiples with a given number of digits. 

In the base-10 case, it is well known that if a 
$(10,k)$-reverse multiple exists 
then $k$ must be 4 or 9. 
The papers by 
Klosinski and Smolarski \cite{KlSm69},
Grimm and Ballew \cite{GrBa75}, and very recently 
Webster and Williams \cite{WeWi13}
show how to find all solutions in this case.
This is especially easy to do using the $(10,4)$ and $(10,9)$
Young graphs (see Figs. \ref{Fig10_4} and \ref{Fig10_9}).

Incidentally, it seems that none of the above authors
noticed that the $n$-digit reverse multiples in the base 10 case
(and in many other cases) are essentially enumerated by
the Fibonacci numbers (see \eqn{GF109e}--\eqn{Eq10} below;
the first mention of this fact appears to have been
by D.~W.~Wilson \cite{Wils97} in 1997, in a comment on
one of the sequences in \cite{OEIS}).

Another property that these authors overlooked
is that in many (but not all) cases the $(g,k)$-reverse multiples
are precisely the numbers $\gamma \beta$,
where $\gamma$ is a constant (depending on $g$ and $k$) and $\beta$
ranges over all numbers whose base-$g$ expansion 
is palindromic, with a restricted set of digits, and satisfies
certain simple rules. In the case of the $(10,9)$-reverse multiples,
for example, $\gamma = 99$ and $\beta$ is any positive
palindromic decimal number whose digits are 0 or 1
and whose decimal expansion does not contain any singleton 0's or 1's
(again this was first noticed by Wilson \cite{Wils97}). 
Similarly, the $(18,7)$-reverse multiples are the numbers $\gamma \beta$
where $\gamma = (1,17,3,5,12,13)_{18}$ and $\beta$ is any
positive number whose base-18 expansion is palindromic, contains only the
digits 0 or 1, and does not contain any run of 0's of length less than 3 
or any pair of adjacent 1's (see \S\ref{Sec4}).
On the other hand, the $(24,17)$- and $(40,13)$-reverse multiples are
not palindromic in this sense (see the last two examples in \S\ref{Sec4}).

The Young graphs are defined in \S\ref{Sec2},
and the transfer-matrix method is described in \S\ref{Sec3}.
Sections \ref{Sec34}, \ref{Sec35}, \ref{Sec36}
discuss three particular families of Young graphs: 
the ``1089'' graph shown in
Figs. \ref{Fig10_4}, \ref{Fig10_9}, which appears to occur
if and only if $k+1$ divides $g$, and is consequently
the most common Young graph; the
complete graphs $K_m$ (see Figs. \ref{Fig5_2}, \ref{Fig11_3});
and the cyclic graphs $Z_m$ illustrated in Figs. \ref{FigCyclic},
\ref{Fig18_7}.
Section \ref{Sec34} explains why Fibonacci numbers
arise from the 1089 graph.
The computer-generated results for bases $g \le 20$ appear in \S\ref{Sec4}
(see especially Table \ref{gkTable}, which lists all Young graphs with $g \le 20$).
The final section lists several open problems.


\vspace*{+.1in}
\noindent
\textbf{Notation.}
A nonnegative number $N$ with base-$g$ expansion
\beql{Eq1}
N ~=~ \sum_{i=0}^{n-1} a_i g^i
\eeq
($0 \le a_i <g$, $a_{n-1} \ne 0$) will
also be written as 
\beql{Eq2}
N ~=~ (a_{n-1}, a_{n-2}, \ldots, a_1, a_0)_g \,.
\eeq
We refer to the $a_i$ as ``digits'', even if $g \ne 10$,
and say that $N$ has ``length'' $n$.
For the convenience of readers who wish to refer to Young's papers,
for the most part this paper uses her notation.

\section{Young Graphs}\label{Sec2}

\subsection{The equations.}
We assume always that $g \ge 3$, $2 \le k < g$.
Suppose $N$, given by \eqn{Eq1} and \eqn{Eq2}, is
a $(g,k)$-reverse multiple, so that
\beql{Eq3}
kN ~=~ (a_0, a_1, \ldots, a_{n-2}, a_{n-1})_g \,.
\eeq
This implies that $a_0 \ne 0$ (since $kN > N$) and
\begin{align}\label{Eq4}
~ & ka_0 \qquad \qquad = a_{n-1} + r_0 \, g, \notag\\
~ & ka_1+ r_0 \qquad = a_{n-2} + r_1 \, g, \notag\\
~ & \qquad \qquad \ldots \qquad \notag\\
~ & ka_i+ r_{i-1} \quad \; = a_{n-1-i} + r_i \, g, \notag\\
~ & \qquad \qquad \ldots \qquad \notag\\
~ & ka_{n-2} + r_{n-3} = a_{1} + r_{n-2} \, g, \notag\\
~ & ka_{n-1} + r_{n-2} = a_0, 
\end{align}
where the ``carry'' digits $r_0, \ldots, r_{n-2}$ satisfy
$0 \le r_i < g$.
Young's approach \cite{Young1}, \cite{Young2} proceeds as follows.
Equations  \eqn{Eq4} imply that
\beql{Eq5}
r_0 ~>~0, \quad r_i ~<~ k \mbox{~for~} i=0,\ldots,n-2
\eeq
(the latter inequality has an easy proof by induction).
We also set $r_{-1}=r_{n-1}=0$.
The equations can be combined in pairs, the $i$-th pair being
\begin{align}\label{Eq6}
~ & ka_i+ r_{i-1} \quad \qquad = a_{n-1-i} + r_i \, g, \notag\\
~ & ka_{n-1-i} + r_{n-2-i} = a_{i} + r_{n-1-i} \, g, 
\end{align}
for $i=0, 1, \ldots, \lfloor \frac{n}{2} \rfloor$. 
If $n$ is odd, the final pair consists of two identical equations.

These equations are solved recursively. 
For $i=0, 1, \ldots$, we suppose we have found
\beql{Eq6a}
a_0, \ldots, a_{i-1}, \quad a_{n-i}, \ldots, a_{n-1}, \quad r_0, \ldots, r_{i-1},
\quad r_{n-1-i}, \ldots, r_{n-2},
\eeq
and we attempt to find all possibilities for the four unknowns
\beql{Eq6b}
a_i, \; a_{n-1-i}, \; r_i, \; r_{n-2-i}.
\eeq
From Eq. \eqn{Eq6} we have 
\begin{align}\label{EqY}
~ & k\boldsymbol{a_i} + r_{i-1} ~\equiv~ \boldsymbol{a_{n-1-i}} \pmod{g}, \notag\\
~ & 0 ~\le~ \boldsymbol{a_i} + r_{n-1-i} \, g - k\boldsymbol{a_{n-1-i}} ~<~ k.
\end{align}
where the unknowns $a_i$ and $a_{n-1-i}$ 
are shown in bold face for emphasis. 
By direct search, one finds all pairs $a_i, a_{n-1-i}$ with
$0 \le a_i <g$, $0 \le a_{n-1-i} <g$ that satisfy \eqn{EqY}.
For each such pair, define
\begin{align}\label{Eq8}
~ & r_i ~=~ \frac{ka_i+ r_{i-1}  - a_{n-1-i}}{g}, \notag\\
~ & r_{n-2-i} ~=~ a_i + r_{n-1-i} \, g - k a_{n-1-i}. 
\end{align}

\subsection{The graph $H(g,k)$.}
We keep a record of the solutions in the form of a finite, directed graph.
Young uses a potentially infinite tree for this,
but a finite graph is more suitable
for enumerating the solutions.
The equations \eqn{EqY} involve $r_{i-1}$ and $r_{n-1-i}$,
but do not involve any of the other variables in \eqn{Eq6a} that 
we are assuming are known. So we take the nodes
of the graph to be the ordered pairs $[r_{n-1-i},r_{i-1}]$.
There are potentially $k^2$ nodes $[r', r]$,
with $0 \le r' < k$, $0 \le r < k$.
For each solution \eqn{Eq6b}, we draw a directed edge from node
$[r_{n-1-i},r_{i-1}]$ to node $[r_{n-2-i},r_{i}]$,
with edge-label $(a_{n-1-i}, a_i)$:
$$
[r_{n-1-i},r_{i-1}] \xrightarrow{(a_{n-1-i}, a_i)} [r_{n-2-i},r_i].
$$
There may be zero, one or several solutions \eqn{Eq6b}, and so
zero, one or several edges emanating from a node.

It may happen that either $a_i$ or $a_{n-1-i}$ (or both) are $0$ 
in a solution to \eqn{EqY}. 
This is allowed only if $i \ne 0$ (since
$N$ may not begin or end with $0$).
To handle this problem, we introduce a special {\em starting node}
$[[0,0]]$ (written with double brackets, and indicated in the drawings by
a slightly larger black circle), where we start the recursion, and
which has the property that no edge emanating from
it can have the label $(0,0)$.
The starting node has no incoming edges, by definition.
If the starting node is included, there are a maximum of
$k^2+1$ nodes in the Young graph. However, it is more natural to
ignore the starting node and the edges emanating from it when 
discussing these graphs. The $(8,5)$ Young graph in Fig.
\ref{Fig8_5}, for example, is best described as an 8-node,
16-edge graph.

The graph is then constructed by beginning at the starting node 
(with $i=0$), and solving \eqn{EqY} recursively, until
we have found all edges that emanate from every node we have encountered.
If $(g,k)$-reverse multiples exist, an internal (non-starting) node $[0,0]$ 
will always occur.
The graph may also contain loops (a loop is an edge
directed from a node to itself). However, there is
at most one edge directed from any one node to another (see \S\ref{Sec27}).

Let $H(g,k)$ denote the resulting directed graph.
This is not yet the Young graph, for there is one further
condition that must be satisfied.  

\subsection{Young's theorem.}
The following is a restatement of
the main result of Young \cite{Young1}.

\begin{theorem}\label{Th1}
(i) A $(g,k)$-reverse multiple exists if and only if the graph $H(g,k)$
contains either a node $[r, r]$ with $r \ne 0$, or
a pair of nodes $[r', r] \rightarrow [r, r']$
joined by an edge, with $r' \ne r$. \\
(ii) The $(g,k)$-reverse multiples with an even 
number $n=2t$  of digits
are in one-to-one correspondence with paths
that go from the starting node
to a node $[r, r]$, where now $r$ may be $0$,
as shown in Fig. \ref{Fig63}. \\
(iii) The $(g,k)$-reverse multiples with an odd 
number $n=2t+1$  of digits
are in one-to-one correspondence with paths that 
\textbf{either} 
go from the starting node to the first of a pair of 
adjacent nodes $[r', r] \rightarrow [r, r']$ with 
$r' \ne r$, as shown in Fig. \ref{Fig73},
\textbf{or} 
go from the starting node to a node $[r, r]$ with 
a loop, where again $r$ may be $0$, as shown in Fig. \ref{Fig72}.
\end{theorem}

\begin{figure}[!h]
\centerline{\includegraphics[width=5.5in]{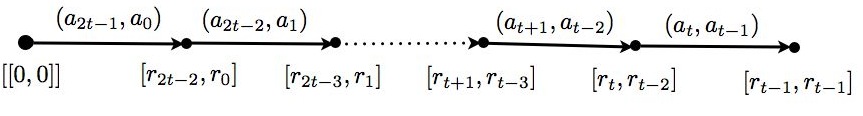}}
\caption{There is a one-to-one correspondence between these paths in the Young
graph and reverse multiples with an even number
of digits; $[r_{t-1},r_{t-1}]$ is an even pivot node.}
\label{Fig63}
\end{figure}
\begin{figure}[!h]
\centerline{\includegraphics[width=5.5in]{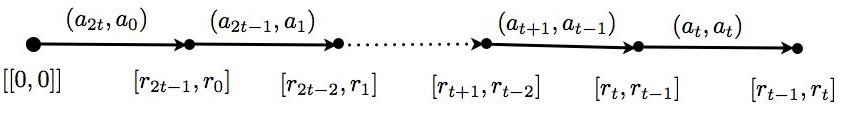}}
\caption{There is a one-to-one correspondence between
the paths shown in this figure and Fig. \ref{Fig72}
and reverse multiples with an odd number
of digits. Here $r_{t} \ne r_{t-1}$;
$[r_{t},r_{t-1}]$ is an odd pivot node.}
\label{Fig73}
\end{figure}
\begin{figure}[htb]
\centerline{\includegraphics[width=5.5in]{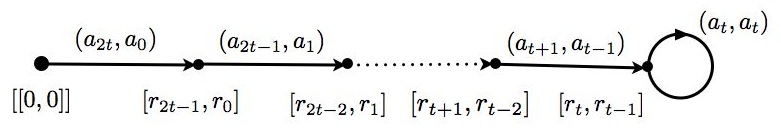}}
\caption{Here $r_t = r_{t-1}$; $[r_{t},r_{t-1}]$ is an odd pivot node.}
\label{Fig72}
\end{figure}

We refer to \cite{Young1} (Theorems 1, 2, 3 and Corollaries 1, 2) for the proof.
We will call any node $[r, r]$ (including the non-starting node $[0,0]$)
an {\em even pivot node}, and
any node $[r, r]$ with a loop (again $r$ may be 0)
or the initial node $[r', r]$ of a pair $[r', r] \rightarrow [r, r']$
with $r' \ne r$, an {\em odd pivot node}.
For examples, see Figs. \ref{Fig10_4}, \ref{Fig10_9}, \ref{Fig8_5},
\ref{Fig11_7}, \ref{Fig14_3}.

\subsection{Young graphs.}
We can now construct the $(g,k)$ Young graph.
For it to exist, the graph $H(g,k)$ must contain
at least one (even {\em or} odd) nonzero pivot node. If so,
there may be some nodes (``dead ends'') 
which can be reached from the starting node but which
do not lead to any pivot node.
These nodes and all edges leading to
them are now removed (or ``pruned'') from 
$H(g,k)$: the result is the $(g,k)$ Young graph.

Figures \ref{Fig10_4} and \ref{Fig10_9}
show the $(10,4)$ and $(10,9)$ Young graphs.
Here no pruning is necessary: these are the graphs
$H(10,4)$ and $H(10,9)$.
On the other hand, the graph $H(8,3)$ (not shown) has seven nodes.
One of them, however, has three incoming edges and no outgoing edges,
and another has two incoming edges and no outgoing edges.
When these two nodes and the five edges are removed, the
result, the $(8,3)$ Young graph, is the same graph 
that underlies Figs. \ref{Fig10_4}
and \ref{Fig10_9} (with yet a third set of labels).
$H(12,7)$ (also not shown) is an example of a graph which contains no pivot nodes, 
so the $(12,7)$ Young graph does not exist.

\begin{figure}[htb]
\begin{minipage}[b]{0.45\linewidth}
\centering
\includegraphics[width=0.6\textwidth]{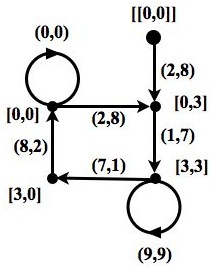}
\caption{The $(10,4)$ Young graph.
$[0,0]$ is both an even pivot node and an odd pivot node; so is $[3,3]$.}
\label{Fig10_4}
\end{minipage}
\hspace{0.5cm}
\begin{minipage}[b]{0.45\linewidth}
\centering
\includegraphics[width=0.6\textwidth]{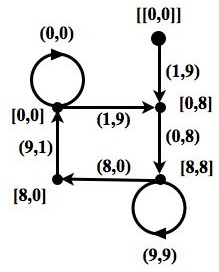}
\caption{The $(10,9)$ Young graph.
$[0,0]$ is both an even pivot node and an odd pivot node; so is $[8,8]$.}
\label{Fig10_9}
\end{minipage} 
\end{figure}


\subsection{Correspondence between paths and $(g,k)$-reverse multiples.}
Once we have the $(g,k)$ Young graph and a list of
the pivot nodes, equations \eqn{Eq4} and \eqn{Eq6}
imply that the correspondence between the 
paths of the three forms shown in Figs. \ref{Fig63}, \ref{Fig73}, \ref{Fig72}
and the $(g,k)$-reverse multiples is as follows.

\textbf{The path in Fig. \ref{Fig63}.}
The right-most $t$ digits $a_0, \ldots, a_{t-1}$
of the corresponding $2t$-digit number $N$ (see \eqn{Eq1})
are obtained by reading the right-hand labels on the edges
as we proceed from the starting node $[[0,0]]$,
following the arrows until we reach the pivot node $[r_{t-1}, r_{t-1}]$.
The remaining $t$ digits $a_{t}, \ldots, a_{2t-1}$ of $N$
are obtained by reading the left-hand labels on the edges
as we return to the starting node, going against the arrows,
and exactly retracing the forward path.

Similarly, the check digits $r_{-1}=0, r_0, \ldots, r_{t-1}$
are obtained by reading the right-hand node labels
along the path from the starting node to the pivot node,
and the check digits $r_{t-2}, \ldots, r_{2t-2}$, $r_{2t-1}=0$
by reading the left-hand node labels
as we return to the starting node going against the arrows.
We use only one copy of the pair of identical labels $r_{t-1}$
on the pivot node.

\begin{figure}[htb]
\centerline{\includegraphics[width=2.5in]{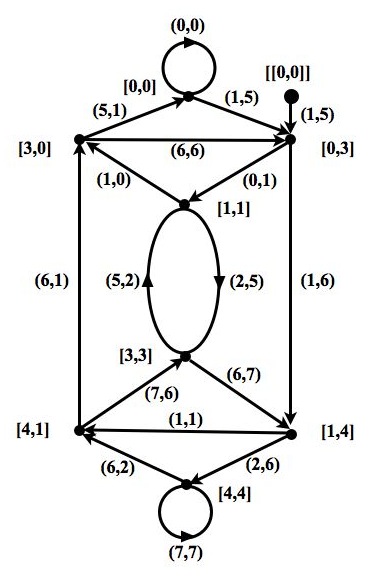}}
\caption{The $(8,5)$ Young graph.
Even pivot nodes: $[0,0]$, $[1,1]$, $[3,3]$, $[4,4]$;
odd pivot nodes: $[0,0]$, $[1,4]$, $[3,0]$, $[4,4]$.}
\label{Fig8_5}
\end{figure}

For example, consider the case $g=8$, $k=5$.
The $(8,5)$ Young graph is shown in Fig. \ref{Fig8_5}.
(In this case $H(8,5)$ {\em is} the Young graph.)
Look at the path 
$[[0,0]] \rightarrow [0,3] \rightarrow [1,1] \rightarrow [3,3]$
to the pivot node $[3,3]$. 
Reading along the edge labels gives the number
$N = (1,0,2,5,1,5)_8$, and we verify that this is an $(8,5)$-reverse multiple
by multiplying it by 5 mod 8:
\begin{center}
$$
\begin{tabular}{ l c p{.001in} p{.001in} p{.001in} p{.001in}  p{.001in}  p{.001in}  p{.001in} }
$N$ & =       &   & 1 & 0 & 2 & 5 & 1 & 5 \\
    &         &   &   &   &   &   & $\times$ & 5 \\
\hline
$5N$ & =      &   & 5 & 1 & 5 & 2 & 0 & 1 \\
\hline
carries  & =  & 0 & 0 & 1 & 3 & 1 & 3 & 0 \\
\end{tabular}
$$
\end{center}
The seven carry digits $r_{-1} = 0, r_0 = 3, r_1 = 1,
\ldots, r_4 = 0, r_5 = 0$ (reading
from right to left) are shown at the bottom of the tableau,
and match the node labels along the path.

\textbf{The path in Fig. \ref{Fig73}.}
Again, we first read the right-hand edge labels
$a_0, \ldots, a_{t-1}$ on the path from the starting node 
to the pivot node $[r_{t}, r_{t-1}]$.
Then we take $a_t$ from the last edge in Fig. \ref{Fig73},
and the remaining digits $a_{t+1}, \ldots, a_{2t}$
are obtained by reading the left-hand labels on the edges
as we return to the starting node, going against the arrows.
We only take one copy of $a_t$ from the last edge in Fig. \ref{Fig73}.

The check digits $r_{-1}=0, \ldots, r_{t-1}, r_t, \ldots, r_{2t}=0$
are obtained by reading the right-hand node labels
along the path from the starting node to the pivot node,
and then reading the left-hand node labels
as we return to the starting node going against the arrows.
The final node in Fig. \ref{Fig73} is not used when we record the check digits.

For example, consider the path
$[[0,0]] \rightarrow [0,3] \rightarrow [1,4] \rightarrow [4,1]$
in Fig. \ref{Fig8_5} where $[1,4]$ is the pivot node.
Reading along the edge labels gives the number
$N = (1,1,1,6,5)_8$, and again we verify that this is an $(8,5)$-reverse multiple:
\begin{center}
$$
\begin{tabular}{ l c p{.001in} p{.001in} p{.001in} p{.001in}  p{.001in}  p{.001in}  p{.001in} }
$N$ & =       &  & 1 & 1 & 1 & 6 & 5 \\
    &         &  &   &   &   & $\times$ & 5 \\
\hline
$5N$ & =      &  & 5 & 6 & 1 & 1 & 1  \\
\hline
carries  & =  & 0 & 0 & 1 & 4 & 3 & 0 \\
\end{tabular}
$$
\end{center}
The six carry digits $r_{-1} = 0, r_0 = 3, r_1 = 4,
r_2 = 1, r_3 = 0, r_4 = 0$ (reading
from right to left) are shown at the bottom of the tableau,
and match the node labels along the path (ignoring the node $[4,1]$).

\textbf{The path in Fig. \ref{Fig72}.}
Once again, we first read the right-hand edge labels
$a_0, \ldots, a_{t-1}$ on the path from the starting node
to the pivot node $[r_{t}, r_{t-1}]$.
Then we take $a_t$ from the loop,
and the remaining digits $a_{t+1}, \ldots, a_{2t}$
are obtained by reading the left-hand labels on the edges
as we return to the starting node, going against the arrows.
We only take one copy of $a_t$ from the loop.

Again, the check digits $r_{-1}=0, \ldots, r_{t-1}, r_t, \ldots, r_{2t}=0$
are obtained by reading the right-hand node labels
along the path from the starting node to the pivot node,
and then reading the left-hand node labels
as we return to the starting node going against the arrows.

We give two examples.
(i) In Fig. \ref{Fig8_5} consider the path
$[[0,0]] \rightarrow [0,3] \rightarrow [1,4] \rightarrow [4,4]$
to the pivot node $[4,4]$, using the loop at $[4,4]$ once.
Reading along the edge labels gives the number
$N = (1,1,2,7,6,6,5)_8$, and we verify that this is an $(8,5)$-reverse multiple:
\begin{center}
$$
\begin{tabular}{ l c 
p{.001in} p{.001in} p{.001in} p{.001in} p{.001in}  p{.001in}  p{.001in}  p{.001in} }
$N$ & =       &   & 1 & 1 & 2 & 7 & 6 & 6 & 5 \\
    &         &   &   &   &   &   &   & $\times$ & 5 \\
\hline
$5N$ & =      &   & 5 & 6 & 6 & 7 & 2 & 1 & 1 \\
\hline
carries  & =  & 0 & 0 & 1 & 4 & 4 & 4 & 3 & 0 \\
\end{tabular}
$$
\end{center}
The eight carry digits $r_{-1} = 0, r_0 = 3, r_1 = 4,
\ldots, r_5 = 0, r_5 = 0$ (reading
from right to left) are shown at the bottom of the tableau,
and match the node labels along the path.
Of course we could have traversed the loop at $[4,4]$ any number of times: this
would simply have increased the number of 7's in the middle of $N$.

(ii) Here is a more elaborate example,
to illustrate the use of the (non-starting) node $[0,0]$ as
the pivot node.
Consider the following path in Fig. \ref{Fig8_5} from the starting node  $[[0,0]]$
to the pivot node $[0,0]$:
$$
[[0,0]] \rightarrow
[0,3] \rightarrow
[1,4] \rightarrow
[4,4] \rightarrow
[4,1] \rightarrow
[3,0] \rightarrow
[0,0] .
$$
We then follow the loop at $[0,0]$ once.
Reading along the edge labels gives the number
$N = (1,1,2,6,6,5,0,1,1,2,6,6,5)_8$,
and again we verify that this is an $(8,5)$-reverse multiple:
\begin{center}
$$
\begin{tabular}{ l c 
p{.001in} p{.001in} p{.001in} p{.001in}  p{.001in}  p{.001in}  p{.001in} 
p{.001in} p{.001in} p{.001in} p{.001in}  p{.001in}  p{.001in}  p{.001in} }
$N$ & =  & &1&1&2&6&6&5&0&1&1&2&6&6&5 \\
    &    & & & & & & & & & & & & &$\times$&5 \\
\hline
$5N$& =  & &5&6&6&2&1&1&0&5&6&6&2&1&1 \\
\hline
carries&=&0&0&1&4&4&3&0&0&0&1&4&4&3&0 \\
\end{tabular}
$$
\end{center}

Returning to the base-10 case, by following the paths from 
the starting node in Fig. \ref{Fig10_4}
to the two pivot nodes, we obtain all the $(10,4)$-reverse multiples,
which are:
\beql{A008918}
2178, 21978, 219978, 2199978, 21782178, 21999978, 217802178, 219999978, \ldots,
\eeq
and from Fig. \ref{Fig10_9} the $(10,9)$-reverse multiples:
\beql{A001232}
1089, 10989, 109989, 1099989, 10891089, 10999989, 108901089, 109999989, \ldots  .
\eeq
See entries A008918, A001232, A008919 in 
the {\em On-Line Encyclopedia of Integer Sequences} \cite{OEIS}
for more terms.
(Six-digit numbers prefixed by A will
always refer to entries in \cite{OEIS}.)
We return to the discussion of these numbers in \S\ref{Sec32} and \S\ref{Sec34}.

\subsection{Summary.}
The $(g,k)$-reverse multiples $N$ 
with an {\em even} number of digits are
in one-to-one correspondence with paths in the $(g,k)$ Young
graph that go from the
starting node $[[0,0]]$ to a pivot node $[r,r]$ by following
the arrows, and then return to the starting node by exactly
retracing the path, only now going against the arrows.
The path may go through intermediate nodes (including pivot nodes that are 
not acting as pivots) any number of times,
and may traverse edges any number of times.

The $(g,k)$-reverse multiples $N$ with an {\em odd} number of digits
have a similar description, only now the outward path ends at
either a pivot node which is 
the first of a pair of adjacent nodes $[r', r] \rightarrow [r, r']$ 
(with $r' \ne r$)
or a pivot node $[r,r]$ with a loop.

The resulting paths of course can get very complicated. The
great merit of the transfer-matrix method (see \S\ref{Sec3}) is that 
it makes it easy to count the paths 
of any length and so to find the number of reverse multiples with 
any given number of digits.

\subsection{Properties of Young graphs.}\label{Sec27}
The following are some useful properties of Young graphs.
Properties (P1) and (P3) were (essentially)
stated and proved by Young \cite{Young1}, \cite{Young2}.

(P1) The label on an edge is determined by the labels
on its two end-nodes. For we can solve \eqn{Eq6} to obtain:
\begin{align}\label{Eq93.5}
a_i &~=~ 
\frac{k r_i g - k r_{i-1} + r_{n-1-i} g - r_{n-2-i}}{k^2-1}, \nonumber \\
a_{n-1-i} &~=~ 
\frac{k r_{n-1-i} g - k r_{n-2-i} + r_{i} g - r_{i-1}}{k^2-1}.
\end{align}

(P2) There is at most one edge between any two nodes. This
follows from (P1).

(P3) If the graph contains an edge 
$[r,s] \xrightarrow{(a,b)} [t,u]$
then it also contains nodes $[s,r]$ and $[u,t]$ and an edge 
$[u,t] \xrightarrow{(b,a)} [s,r]$.
Sketch of proof: 
Let $N = (a_{n-1}, \ldots, a_0)_g$ be a $(g,k)$-reverse multiple for
which the corresponding path contains the first edge.
Then $(a_{n-1}, \ldots, a_0, a_{n-1}, \ldots, a_0)_g$ 
is also a $(g,k)$-reverse multiple, and its path consists of
the path for $N$ followed by that same path with all 
the node-labels, edge-labels and arrows reversed, and
therefore contains the second edge.

(P4) The graph always contains a (non-starting) node $[0,0]$.
This follows from (P3), since we have the node $[[0,0]]$.

(P5)  If the graph contains an edge
$[r',r] \xrightarrow{(a,b)} [r,r']$ with $r' \ne r$, or
if a node $[r,r]$ has a loop with label $(a,b)$,
then $a=b$. Again this follows at once from (P3).

(P6) It follows from the definition of the Young graph 
that the node labels are all distinct. 
The edge labels are also all distinct. For suppose
on the contrary that the same label
appeared on two different edges:
$$
[r_{n-1-i},r_{i-1}] \xrightarrow{(a_{n-1-i}, a_i)} [r_{n-2-i},r_i],
\quad
[s_{n-1-i},s_{i-1}] \xrightarrow{(a_{n-1-i}, a_i)} [s_{n-2-i},s_i].
$$
From \eqn{Eq93.5} we have
\begin{gather}
k r_i g - k r_{i-1} + r_{n-1-i} g - r_{n-2-i}
=
k s_i g - k s_{i-1} + s_{n-1-i} g - s_{n-2-i}, \notag \\
k r_{n-1-i} g - k r_{n-2-i} + r_{i} g - r_{i-1}
=
k s_{n-1-i} g - k s_{n-2-i} + s_{i} g - s_{i-1}.\notag
\end{gather}
If we collect terms, and subtract the second equation from 
$k$ times the first equation, we obtain
$g(r_i-s_i) = r_{i-1}-s_{i-1}$. If $r_i \ne s_i$, the left-hand
side is greater than $g$ in magnitude, while the right-side is
not. So $r_i=s_i$, $r_{i-1}=s_{i-1}$. Similarly we obtain
$r_{n-1-i}=s_{n-1-i}$
and
$r_{n-2-i}=s_{n-2-i}$. 

(P7) Sutcliffe \cite[Theorem~2]{Sut66} shows that for $g \ge 4$,
there is a two-digit $(g,k)$-reverse multiple for some $k$
if and only if $g+1$ is composite. This corresponds to an edge 
$[0,0] \xrightarrow{(b,a)} [r,r]$ with 
$a = r(kg-1)/(k^2-1)$, $b = r(g-k)/(k^2-1)$.
(See also Conjecture \ref{Conj24} in \S\ref{Sec35}.)

\subsection{Further remarks.}\label{SecRem}
(i) The only difference between the $(10,4)$
and $(10,9)$ Young graphs (shown in Figs. \ref{Fig10_4} and 
\ref{Fig10_9}) is in the labels on the nodes and edges. The underlying
abstract directed graphs are the same and the pivot nodes
are the same. In such a case we
say that the two Young graphs are {\em isomorphic}
(see also the first open question in \S\ref{Sec5}).

(ii) Some of the references mention that fact that in base 10,
reverse multiples are never prime (for example,
$1089 = 3^2 \times 11^2$).
In fact, no $(g,k)$-reverse multiple $N$ can ever be prime,
for if we add all the
equations \eqn{Eq4} and read the result mod $g-1$, we find that
$(k-1) \sum_{i=0}^{n-1} a_i \equiv 0 \pmod{g-1}$.
This implies that $N$ (see \eqn{Eq1}) is a multiple of $g-1$
(in the case $g=10$, this is ``casting out 9's'').


\section{Counting Reverse Multiples}\label{Sec3}
\subsection{Adjacency matrices.}\label{Sec32}
Let $c_{t}$ denote the number of $(g,k)$-reverse multiples with $t$ digits,
with generating function
\beql{EqGF3}
\sC (x) \, =\,  \sum_{t \ge 0} c_{t} \, x^{t},
\eeq
and let
$\sP (x) \, =\, \sum_{t \ge 0} c_{2t} \, x^{2t}$,
$\sQ (x) \, =\, \sum_{t \ge 0} c_{2t+1} \, x^{2t+1}$, 
with
$\sC (x) \, =\,  \sP(x)\, +\,  \sQ(x)$.


We will calculate these generating functions
using the transfer-matrix method \cite[\S4.7]{Stanley}.
Suppose the $(g,k)$ Young graph has $v$ nodes (including
the starting node) and $e$ edges.
We label the nodes $V_0, V_1, \ldots, V_{v-1}$,
where we take $V_0$ to be the starting node $[[0,0]]$
and $V_1$ to be the node $[0,0]$.
We then construct a $v \times v$ adjacency matrix $A$
for the graph, by setting $A\left|_{i,j} \right.$
equal to $x^2$ if there is a directed edge from $V_i$ to $V_j$, and 0 otherwise
(the entry is $x^2$ rather than $x$, because each edge
is used twice, once in the direction of the arrow,
once going against the arrow).

The key observation (this is the essence
of the transfer-matrix method) is that if
$V_i= [r,r]$ is a pivot node, the 
$(V_0, V_i)$ entry of $A^t$ will be $mx^{2t}$ if there are $m$
paths of length $t$ from the starting node to $V_i$,
which by Theorem \ref{Th1} 
means that there are $m$ $(g,k)$-reverse multiples
of length $2t$ for which the corresponding
path has pivot node $V_i$.

The sum of the entries $(V_0, V_i)$ in $A^t$ over all pivot nodes $V_i = [r,r]$
is therefore equal to $c_{2t} x^{2t}$.
Summing on $t$, we see that the generating function $\sP (x)$ is the
sum over all pivot nodes $V_i = [r,r]$ of 
the $(V_0, V_i)$ entries in the matrix
$$
B ~=~ A + A^2 + A^3 + \cdots ~=~ A(I-A)^{-1}.
$$

Similarly, $c_{2t+1} x^{2t+1}$ is equal to $x$ times the sum
of the entries $(V_0, V_i)$ in $A^t$ 
for which $V_i$ is a pivot node for a reverse multiple
with an odd number of digits
(meaning $V_i$ is either a node $[r,r]$ with a loop
or the first of a pair of adjacent nodes $[r', r] \rightarrow [r, r']$
with $r' \ne r$).
The generating function $\sQ (x)$ is the
sum of the $(V_0, V_i)$ entries in $B$
over all such pivot nodes.

For example, consider the $(10,9)$ Young graph in
Fig. \ref{Fig10_9}. Let $V_2$, $V_3$, $V_4$
denote the nodes $[0,8]$, $[8,8]$, $[8,0]$ respectively.
The adjacency matrix $A$ is
$$
\left[\begin{array}{ccccc}
0 & 0 & x^2 & 0 & 0 \\
0 & x^2 & x^2 & 0 & 0 \\
0 & 0 & 0 & x^2 & 0 \\
0 & 0 & 0 & x^2 & x^2 \\
0 & x^2 & 0 & 0 & 0 \\
\end{array}
\right]\,.
$$
The pivot nodes $[r,r]$ are $V_1$ and $V_3$, and any computer algebra
system will tell us that 
\begin{align}
B\left|_{0,1} \right. ~=~ & \frac{x^4(1-x^2)}{1-2x^2+x^4-x^8},  \notag \\
B\left|_{0,3} \right. ~=~ & \frac{x^8}{1-2x^2+x^4-x^8},
\notag
\end{align}
whose sum is the generating function for
$(10,9)$-reverse multiples with an even number of digits:
\begin{align}\label{GF109e}
\sP (x) & ~=~ \frac{x^4}{1-x^2-x^4} \notag \\
        & ~=~ \sum_{t=2}^{\infty} \, F_{t-1}x^{2t} \notag  \\
        & ~=~ x^4 + x^6 + 2x^8 + 3x^{10} + 5x^{12} + 8x^{14} + \cdots, 
\end{align}
where $F_n$ is the $n$-th Fibonacci number.

Both $V_1$ and $V_3$ are also pivot nodes for reverse multiples with 
an odd number of digits,
and so $\sQ (x) = x \sP (x)$, and
\begin{align}\label{GF109}
\sC (x) & ~=~ \sP(x) + \sQ(x) \notag \\
        & ~=~ \frac{x^4(1+x)}{1-x^2-x^4} \notag \\
        & ~=~ \sum_{t=4}^{\infty} \, F_{\lfloor \frac{t}{2} \rfloor -1}x^{t} \notag \\
        & ~=~ x^4 + x^5 + x^6 + x^7 + 2x^8 + 2x^9 + 3x^{10} + 3x^{11} + 5x^{12} + 5x^{13} + \cdots. 
\end{align}
The coefficients form entry A103609 in \cite{OEIS}.

Isomorphic Young graphs have the same generating function, so the
generating function for the number of $(10,4)$-reverse multiples
is also given by \eqn{GF109}.

If we are only interested in the number of base-10 reverse multiples
with $t$ digits, regardless of the multiplier,
the generating function is
\beql{Eq10}
\frac{2x^4(1+x)}{1-x^2-x^4} 
~=~ 2x^4 + 2x^5 + 2x^6 + 2x^7 + 4x^8 + 4x^9 + 6x^{10} + 6x^{11} + 10x^{12} + 10x^{13} + \cdots,
\eeq
with coefficients $2F_{\lfloor \frac{t}{2} \rfloor -1}$ 
that are twice Fibonacci numbers. The first 2 refers to the two 
numbers in the Hardy story mentioned in
the opening paragraph. 
Would this generating function have changed
Hardy's opinion of the problem?

\subsection{The ``1089'' graph.}\label{Sec34}
We call any Young graph that is isomorphic to the
$(10,4)$ and $(10,9)$ graphs of Figs. \ref{Fig10_4}
and \ref{Fig10_9}  a {\em 1089 graph}.
As can be seen from the entries labeled $a$ in Table \ref{gkTable} of \S\ref{Sec4},
there are many occurrences of this graph.
If true, the following conjecture would explain all these cases.

\begin{conj}\label{Conj1089}
A necessary and sufficient condition for
the $(g,k)$ Young graph to be isomorphic to the 1089
graph is that $k+1$ divides $g$.
\end{conj}

The conjecture is certainly true for $g \le 40$. Checking the conjecture 
for larger values of $g$ is laborious, because many of the larger graphs 
$H(g,k)$ require extensive pruning before the Young graph can be identified.

The next theorem characterizes the reverse multiples 
arising from the 1089 graph.
\begin{theorem}\label{Th1089}
If $k+1$ divides $g$, and the $(g,k)$ Young graph is the 1089 graph, 
let $b = g/(k+1)$. 
(i) The labels on the Young graph are obtained 
by replacing the numbers 1, 2, 3, 7, 8 and 9 on
the $(10,4)$ Young graph shown in Fig. \ref{Fig10_4}
by $b-1$, $b$, $k-1$, $kb-1$, $kb$ and $g-1$, respectively.
The two pivot nodes have labels $[0,0]$ and $[k-1,k-1]$. 
(ii) The generating function $\sC(x)$ is given by \eqn{GF109}.
(iii) The $(g,k)$-reverse multiples are all the numbers
of the form $\gamma \beta$,
where $\gamma = b(g^2-1) = (b-1,g-1,kb)_g$ and 
$\beta$ is any positive number whose base-$g$ expansion
is palindromic, contains only the digits $0$ and $1$,
and does not contain any single $0$'s or $1$'s 
(i.e., any run of consecutive equal digits must have length at least $2$).
(iv) The shortest reverse multiple has length 4, and there are
$F_{\lfloor \frac{t}{2} \rfloor -1}$
reverse multiples of length $t \ge 4$.
\end{theorem}

\begin{proof}
(i) Let $N_0 = (b,b-1,kb-1,kb)_g$.
The following calculation shows that $N_0$ is a $(g,k)$-reverse multiple,
and via Theorem \ref{Th1} gives the labels on the Young graph:
\begin{center}
$$
\begin{tabular}{ l c p{.1in} p{.2in}  p{.4in}  p{.4in}  p{.3in} }
$N_0$ & =       &   & $b$ & $b-1$ & $kb-1$ & $kb$ \\
    &         &   &   &   &  & $\times k$ \\
\hline
$kN_0$ & =      &   & $kb$ & $kb-1$ & $\,b-1$ & $b$ \\
\hline
carries  & =  & 0 & 0 & $k-1$ & $k-1$ &0 \\
\end{tabular}
$$
\end{center}
(ii) The generating function (which does not depend on the labels)
was calculated in \S\ref{Sec32}.
(iii) Let $W_n$ denote the number
$$
(b,b-1,g-1,g-1,\ldots,g-1,g-1,kb-1,kb)_g,
$$
with $n-4$ copies of $g-1$ in the center of the number
(so $W_4 = N_0$).
Inspection of the possible paths in the
1089 graph from the starting node to either of the pivot nodes
shows that every $(g,k)$-reverse multiple
has a unique representation
as a nonempty `word' over the `alphabet' 
$\{0, W_4, W_5, W_6, \ldots \}$,
using the terminology of formal language theory, cf. \cite{Loth}.
For the path to the first pivot gives $W_4$. 
Going around the loop at this pivot several times
changes $W_4$ to $W_5$, $W_6, \ldots .$
Proceeding to the next pivot introduces another $W_4$,
going around the loop at the second pivot 
introduces 0's into the number, and so on.
The reverse multiples of lengths 4 through 10 are
$W_4$, \ldots, $W_8$,
$W_9$, $W_4 0 W_4$, $W_{10}, W_5^2, W_4 0^2 W_4$.

It is easy to see that $W_n = \gamma \times (1,1,\ldots,1)_{g}$, with $n-2$ 1's,
where $\gamma = (b-1,g-1,kb)_g$,
$W_n 0 W_n = \gamma \times (1,1,\ldots, 1,0,0,1,\ldots,1)_g$,
and so on.
We conclude that
any $(g,k)$ reverse multiple is equal to $\gamma \times \beta$,
where $\beta$ is a positive number whose base-$g$ expansion contains only
0's and 1's, is palindromic, and consists of
strings of 1's of length at least 2, separated by strings of
0's of length at least 2.
(iv) This follows at once from the Taylor series expansion of the generating function.
\end{proof}

We can now give a direct explanation for why the Fibonacci numbers
arise in this problem.
It is enough to consider the reverse multiples of even length,
with generating function \eqn{GF109e}.
Let $u_{n}$ denote the number of length-($2n$) choices for
$\beta$ . It is easy to see
that $u_n=u_{n-1}+u_{n-2}$ for $n \ge 3$.
(Look at
the left-hand halves $u_1 \ldots u_n$ of the strings $\beta$.
The strings of length $n$ are obtained by mapping
strings $u_1 \ldots u_{n-2}$ of length $n-2$
to $u_1 u_2 \ldots u_{n-2} u_{n-2} (1-u_{n-2})$, 
and strings $u_1 \ldots u_{n-1}$ of length $n-1$
to $u_1 u_2 \ldots u_{n-1} u_{n-1}$.)
Since this is also the recurrence for the Fibonacci numbers, 
we have $u_n = F_n$.

As an illustration, consider
the case $g=10$. The divisors of 10 are 1, 2, 5 and 10,
and so the conjecture asserts, correctly, that for
$k=4$ and $k=9$ the Young graph is the graph under discussion.
When $k=4$, we have $b=10/5 = 2$, $\gamma = 2(10^2-1) = 198$,
and the theorem asserts, again correctly, that the $(10,4)$-reverse
multiples, given in \eqn{A008918}, are 198 times
one of the base-10 numbers
\beql{A061851}
11, 111, 1111, 11111, 110011, 111111, 1100011, 1111111, 11000011, 11100111, \ldots
\eeq
that are palindromic and contain no singleton 0's or 1's (see A061851). 
Similarly, when $k=9$, $b=1$, $\gamma = 99$,
and the $(10,9)$-reverse multiples, given in \eqn{A001232},
are 99 times the numbers in \eqn{A061851}.

Conjecture \ref{Conj1089} is certainly true when $k=g-1$:

\begin{theorem}\label{Thgm1}
$(g,g-1)$-reverse multiples exist for every $g \ge 3$,
and the $(g,g-1)$ Young graph is the 1089 graph.
\end{theorem}

\begin{proof}
(Sketch.) It is easy to find all solutions to \eqn{EqY}
when $k=g-1$.
We must start with $a_{n-1}=1$, $a_0=g-1$ (or else
$(k-1)a_{n-1}$ would exceed $g$).
This implies $r_0=g-2$, $r_{n-2}=0$.
Next, with $i=1$, the equations \eqn{EqY} become
$a_1+a_{n-2}+2 \equiv 0 \pmod{g}$,
$0 \le a_1 - (g-1)a_{n-2} < g-1$.
If $a_{n-2} \ge 1$ then $a_1 \ge g-1$, implying $a_1 = g-1$,
and the equivalence has no solution. Therefore $a_{n-2}=0$,
which implies $a_1=g-1$ and then $r_1=r_{n-3}=g-2$.
Continuing in this way we find that the graph in
Fig. \ref{Fig10_9} contains all possible solutions.
We omit the details.
\end{proof}

The next result is related to Conjecture \ref{Conj1089}.

\begin{theorem}\label{Thgm1bis}
If $N$ (given by \eqn{Eq1}) is a $(g,k)$-reverse multiple
then $a_0 + a_{n-1} \le g$.
If $a_0 + a_{n-1} = g$ then $k+1$ divides $g$.
\end{theorem}
\begin{proof}
The 0-th pair in \eqn{Eq6} states that
\beql{Eq77}
k a_0 = a_{n-1}+r_0 g, \quad ka_{n-1}+r_{n-2} = a_0\,.
\eeq
Hence $(k-1)(a_0+a_{n-1}) = r_0g-r_{n-2}$,
and so, since $r_0 \le k-1$,
$$
a_0 + a_{n-1} ~=~ \frac{r_0 g -r_{n-2}}{k-1} ~\le~g\,.
$$
If equality holds then $r_0=k-1$,
$r_{n-2}=0$, and \eqn{Eq77} gives $(k^2-1)a_{n-1} = g(k-1)$,
and so $k+1$ divides $g$.
\end{proof}

If true, Conjecture \ref{Conj1089} would imply that the converse holds: if
$k+1$ divides $g$, then any $(g,k)$-reverse multiple
satisfies $a_0 + a_{n-1} = g$ (since then, by Theorem \ref{Th1089}, $a_0=kb$
and $a_{n-1}=b$).

The next conjecture would follow from
Conjecture \ref{Conj1089} and Theorem \ref{Thgm1bis}.
We state it as a separate conjecture since it will be used in 
\S\ref{Sec36}.
It may also have a proof independent of Conjecture \ref{Conj1089}.

\begin{conj}\label{Conj1089a}
If there is a (g,k)-reverse multiple with $a_0 + a_{n-1} = g$
then the Young graph is the 1089 graph.
If the graph is the 1089 graph then every reverse multiple satisfies
$a_0 + a_{n-1} = g$.
\end{conj}

\subsection{The complete Young graphs $K_m$.}\label{Sec35}

Another frequently occurring Young graph is
the complete graph $K_m$.
This is the complete directed graph on $m$ nodes,
with a directed edge from every node to every node,
for a total of $m^2$ edges, together with
the starting node $[[0,0]]$ which is
connected to every node except $[0,0]$.
Every node (except the starting node) has a loop
and is both an even and odd pivot.
The $(5,2)$ Young graph $K_2$
and the $(11,3)$ Young graph $K_3$ are
shown in Figs. \ref{Fig5_2} and \ref{Fig11_3}.

\begin{figure}[htb]
\begin{minipage}[b]{0.45\linewidth}
\centering
\includegraphics[width=1.0\textwidth]{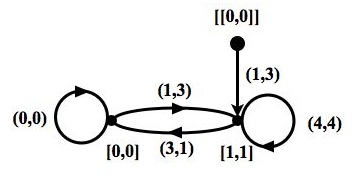}
\caption{The $(5,2)$ Young graph,
the complete graph $K_2$ on two nodes plus the starting node.}
\label{Fig5_2}
\end{minipage}
\hspace{0.5cm}
\begin{minipage}[b]{0.45\linewidth}
\centering
\includegraphics[width=1.0\textwidth]{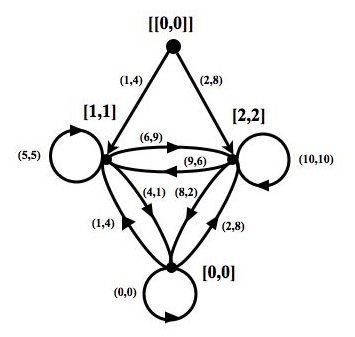}
\caption{The $(11,3)$ Young graph,
the complete graph $K_3$ on three nodes plus the starting node.}
\label{Fig11_3}
\end{minipage}
\end{figure}

As a second illustration of the transfer-matrix method,
consider the $(5,2)$ graph $K_2$.
The adjacency matrix $A$ (taking the nodes in the order $[[0,0]]$,
$[0,0]$, $[1,1]$) is
\beql{EqG2a}
\left[\begin{array}{ccc}
0 & 0 & x^2 \\
0 & x^2 & x^2 \\
0 & x^2 & x^2 
\end{array}
\right]\,.
\eeq
Then
\begin{align}\label{GF52e}
\sP (x) &= B\left|_{0,1} \right. + B\left|_{0,2} \right. \notag \\
        &= \frac{x^4}{1-2x^2} + \frac{x^2(1-x^2)}{1-2x^2} \notag \\
        &= \frac{x^2}{1-2x^2} \notag \\
        &= x^2 + 2x^4 + 4x^6 + 8x^8 + 16x^{10} + \cdots, 
\end{align}
and
\begin{align}\label{GF52}
\sC (x) &= (1+x) \sP(x) ~=~ \frac{x^2(1+x)}{1-2x^2} \notag \\
        &= x^2 + x^3 + 2x^4 + 2x^5 + 4x^6 + 4x^7 + 8x^8 + 8x^{9} + \cdots. 
\end{align}
In this case it is easy to explain the coefficients:
at each step in a path there are two choices for the next node.

The following result gives 
a partial characterization of those Young
graphs that are complete graphs.

\begin{theorem}\label{ThC}
A necessary and sufficient condition for a
Young graph to be a complete graph is that
every node label has the form $[r,r]$.
\end{theorem}
\begin{proof}
(i) Suppose the Young graph is a complete graph, but there is a node
labeled $[r',r]$ with $r' \ne r$.
By (P5), the edge from $[r',r]$ to $[r,r']$ has label $(i,i)$ for some $i$,
and the edge from $[r,r']$ to $[r',r]$ is also labeled $(i,i)$.
Suppose the edge from $[[0,0]]$ to $[r',r]$ is labeled $(j,h)$,
and consider the path
$$
[[0,0]] \xrightarrow{(j,h)} [r',r] \xrightarrow{(i,i)} [r,r'].
$$
Taking the right-hand node $[r,r']$ to be the pivot, we see that
$N = (j,i,i,i,h)_g$ is a $(g,k)$-reverse multiple.
Equations \eqn{Eq6} imply 
$kh=j+rg$,
$ki+r = i+r'g$, 
$ki+r' = rg+i$, 
and the latter two equations imply $(r'-r)g=r-r'$,
so $r=r'$, a contradiction.

(ii) Conversely, suppose all node labels have the form $[r,r]$.
Since the graph is connected, there is an edge
$[[0,0]] \xrightarrow{(b,a)} [r,r]$ for some $r$.
From \eqn{Eq93.5},
$a = r(kg-1)/(k^2-1)$,
$b = r(g-k)/(k^2-1)$,
which imply $a+b = r(g-1)/(k-1) \le g-1$, and so
$N = (b,a+b,a)_g$ is a $(g,k)$-reverse multiple,
and therefore there is a loop at $[r,r]$ with edge-label
$(a+b,a+b)$.
Similar arguments show that if there are edges 
$[[0,0]] \xrightarrow{(b,a)} [r,r]$ and 
$[[0,0]] \xrightarrow{(d,c)} [s,s]$,
then there is an edge 
$[r,r] \xrightarrow{(a+d,b+c)} [s,s]$; and
if there is a pair of edges 
$[[0,0]] \xrightarrow{(b,a)} [r,r] \xrightarrow{(f,e)} [s,s]$,
then there is an edge 
$[[0,0]] \xrightarrow{(f-a,e-b)} [s,s]$.
Iterating these arguments shows that there is a directed 
edge between every pair
of nodes, and every node has a loop.
\end{proof}

If true, the following conjecture,
would explain when the complete graph $K_m$ occurs:

\begin{conj}\label{Conj24}
A necessary and sufficient condition for the $(g,k)$ Young
graph to be the complete graph $K_m$ is that there are exactly
$m-1$ distinct integers $r$ such that
\beql{Eqba}
a = r \frac{kg-1}{k^2-1} < g {~and~}
b = r \frac{g-k}{k^2-1}
\eeq
are positive integers.
\end{conj}

\begin{theorem}\label{ThC2}
(i) If the conditions of Conjecure \ref{Conj24} hold
and the $(g,m)$ Young graph is the complete graph $K_m$,
then (ignoring the starting node) the node labels
are $[ir_0, ir_0]$, $0 \le i \le m-1$, where $r_0$ is
the smallest of the $m-1$ values of $r$.
The edge from the starting node (or from the
node $[0,0]$) to $[ir_0, ir_0]$ has label
$(i r_0 (g-k)/(k^2-1), i r_0 (kg-1)/(k^2-1))$.
(ii) The generating function $\sC(x)$ is
\beql{EqGFK2}
\frac{(m-1)x^2(1+x)}{1-mx^2} ~=~ (m-1)(x^2+x^3+mx^4+mx^5+
m^2x^6+m^2x^7+m^3x^8+\cdots) \,.
\eeq
(iii) The $(g,k)$-reverse multiples are all numbers of the form $\gamma \beta$,
where
$\gamma = (r_0 (g-k)/(k^2-1), r_0 (kg-1)/(k^2-1))_g$
and $\beta$ is any positive number
whose base-$g$ expansion is palindromic and contains only
the digits $0,1,\ldots,m-1$.
\end{theorem}
\begin{proof} (Sketch)
(i) The only thing to be proved is that the values of $r$ are 
multiples of the smallest value, $r_0$,  but this is easily shown.
(ii) This follows from the adjacency matrix, analogous to the proof of \eqn{GF52}.
(iii) From the Young graph we know $\gamma$ is the smallest reverse multiple
and that all the two-digit reverse multiples are the numbers $i \gamma$ 
for $1 \le i \le m-1$.
If we write $\gamma = (b,a)_g$, then as in the proof of Theorem
\ref{ThC} we have $(m-1)(a+b) < g$.
This means that when we calculate $\gamma \beta$ (where
$\beta$ is as in the statement of the theorem)
by `long multiplication', there
are no carries.
We know $k \gamma = \Reverse(\gamma)$.
Since there are no carries, $k \gamma \beta = (k \gamma) \beta
= \Reverse(\gamma) \beta = \Reverse(\gamma \beta)$,
showing that $\gamma \beta$ {\em is} a reverse multiple.
On the other hand a counting argument shows that the 
generating function for the number of $\beta$'s of given length is 
equal to \eqn{EqGFK2} with the factor $x^2$ in the numerator
replaced by $x$.
The generating function for the $\gamma \beta$'s therefore coincides with 
\eqn{EqGFK2}, and it
follows that we have found all the reverse multiples.
\end{proof}

It appears that the first occurrence of $K_m$ 
is when $g=m^2+m-1$ and $k=m$ (and the values
of $r$ in \eqn{Eqba} are simply 1 through $m-1$).
This is certainly true for $m \le 9$, and explains
the $(5,2)$, $(11,3)$, $(19,4)$, $(29,5)$, $(41,6), ~\ldots$
Young graphs.

We end this section with a conjecture which is
much stronger than the result in Theorem \ref{ThC}.
It is, however, consistent with all the data.

\begin{conj}\label{ConjC2}
A necessary and sufficient condition for a 
Young graph to be a complete graph is that there is at least  one
node label $[r, r]$, $r \ne 0$.
\end{conj}

\subsection{The cyclic Young graphs $Z_m$.}\label{Sec36}
The complete graphs described in the previous section
have the maximum number of edges for the given number of nodes.
In the other direction, if there are $m \ge 3$ nodes (not counting
the starting node) then it appears
that the {\em minimum} number of edges in any Young graph is $m+2$.
This is certainly true for $g \le 100$.

Furthermore, it appears that a Young graph
with $m \ge 3$ nodes and $m+2$ edges is always the
``cyclic'' graph $Z_{m-1}$, which we define to
consist of a directed cycle of $m-1$ nodes,
in which one of the edges is paralleled by 
a path of length two passing through the node $[0,0]$,
where there is a loop. We also define $Z_1$ to be $C_2$,
the two-node complete graph in Fig. \ref{Fig5_2}.
The graphs $Z_2$, $Z_3$, $Z_4$ are shown in Fig. \ref{FigCyclic}
and $Z_5$ in Fig. \ref{Fig18_7}.
In $Z_m$, the node $[0,0]$ is always
an odd pivot. 
If $m>1$ is odd there are two even pivot nodes, namely
$[0,0]$ and the node half-way around the cycle (e.g., 
$[8,8]$ in $Z_3$).
If $m$ is even, $[0,0]$ is the only even pivot node
and there is a second odd pivot half-way around
the cycle (e.g., $[5,10]$ in $Z_4$).

\begin{figure}[htb]
\centerline{\includegraphics[width=5.3in]{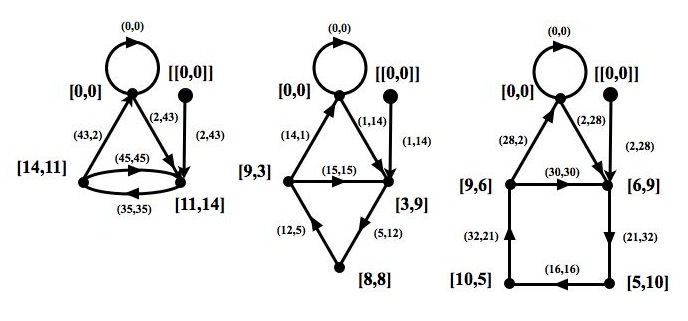}}
\caption{Cyclic (49,16), (17,11) and (34,11) Young graphs 
$Z_2$, $Z_3$ and $Z_4$.}
\label{FigCyclic}
\end{figure}

These cyclic graphs appear to be harder to analyze than the
complete graphs: we have no
analog of Conjecture \ref{Conj24} and
only a conjectural analog of Theorem \ref{ThC2}.
The sequence of pairs $(g,k)$ where the cyclic graphs 
$Z_1$, $Z_2, \ldots$, $Z_9$ first occur
is
\beql{EqFirst}
(5,2), (49,16), (17,11),
(34,11), (18,7), (33,14), (49,39), (77,46),
(63,40), 
\eeq
which has no apparent pattern.

\begin{conj}\label{ConjK1}
(i) For $m \ge 1$, there exists a pair $(g,k)$ for which the Young graph is $Z_m$.
(ii) If the $(g,k)$ Young graph is $K_m$, then the generating function
$\sC(x)$ is
\beql{GFCyclic}
\sC(x) ~=~ \frac{x^{m+1}(1+x)(1-x+x^m)}{1-x^2-x^{2m}}\,.
\eeq
(iii) Let $\gamma$ denote the smallest reverse multiple,
of length $m+1$. 
The $(g,k)$-reverse multiples are all numbers of the form $\gamma \beta$,
where $\beta$ is any positive number
whose base-$g$ expansion is palindromic and contains only
the digits $0$ and $1$, does not contain $11$, and in which any run
of consecutive $0$'s has length at least $m-1$.
\end{conj}

Parts (ii) and (iii) would be a theorem if Conjecture \ref{Conj1089a} were known to be true. 
The proof would be similar to that of Theorem \ref{ThC2}.
We obtain the generating function from the adjacency matrix,
which is a bordered circulant matrix.
The smallest reverse multiple $\gamma$ can be read off the graph.
The difficulty lies in showing that $\gamma \beta$ is a reverse
multiple: for this we need to know that the sum of the first and
last digits of $\gamma$ is less than $g$, which would follow
from Theorem \ref{Thgm1bis} and Conjecture \ref{Conj1089a}.
The final step in the proof uses a simple counting argument to show
that if $u_n$ denotes the number of length-$n$ choices for $\beta$, then 
$u_n = u_{n-2}+u_{n-2m}$, giving a generating function 
which coincides with \eqn{GFCyclic}, and showing that we have indeed found
all the reverse multiples. 
(The sequence $\{u_n\}$ is C-finite, and the `C-finite Ansatz' \cite{KPCT}, 
\cite{Zeil} makes this final step routine.)

\section{Results for Bases $g \le 20$.}\label{Sec4}
A computer was used to find all values of $k$ for which
a $(g,k)$-reverse multiple exists, for $g \le 100$,
and (with one exception) to determine the generating functions
$\sP(x)$, $\sQ(x)$, $\sC(x)$ for the
numbers of reverse multiples.
The calculations were carried out using Maple 16.
The main program
computed the graphs $H(g,k)$, but did not do
any pruning, since the dead ends do not affect the generating functions.
Up to $g=100$, the two largest graphs are $H(58,45)$,
with 588 nodes and 640 edges, and $H(99,68)$,
with 784 nodes and 848 edges.
The program was able to compute the generating functions for all
except $H(99,68)$, where it was unable to compute the matrix $B$.

Two other programs were
used to prune the graphs (for $g \le 40$)
and to make a list of the first 50 or so
reverse multiples by following the paths through the graph.
(Although there is an algorithm for enumerating the paths through a directed
graph \cite{BSTJ}, it does not work well on graphs
with as many circuits as these Young graphs.)

\begin{figure}[htb]
\begin{minipage}[b]{0.45\linewidth}
\centering
\includegraphics[width=0.87\textwidth]{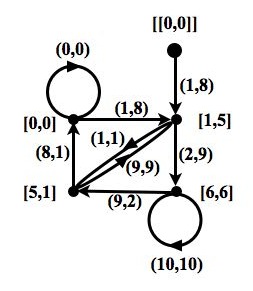}
\caption{The $(11,7)$ Young graph. All nodes are pivot nodes.}
\label{Fig11_7}
\end{minipage}
\hspace{0.5cm}
\begin{minipage}[b]{0.45\linewidth}
\centering
\includegraphics[width=\textwidth]{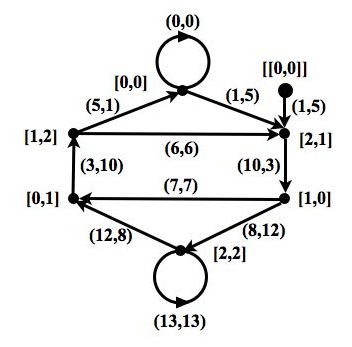}
\caption{The $(14,3)$ Young graph.
Even pivot nodes: $[0,0]$, $[2,2]$; odd pivot nodes: $[0,0]$,
$[1,0]$, $[1,2]$, $[2,2]$.}
\label{Fig14_3}
\end{minipage}
\end{figure}

The results for $g \le 20$ are summarized in Table \ref{gkTable}.
For an extended version of this table, see entry A222817 in \cite{OEIS};
the number of values of $k$ for each $g$ is given in A222819.

\renewcommand{\arraystretch}{1.2}
\begin{table}[htb]
$$
\begin{array}{|@{~}c@{~}|@{~}c@{~}||@{~}c@{~}|@{~}c@{~}|} \hline
g & k                                 & g & k \\ \hline
3 & 2^a                               & 12 & 2^a ,\,3^a ,\,5^a ,\,11^a \\
4 & 3^a                               & 13 & 5^b ,\,6^b ,\,12^a \\
5 & 2^b ,\,4^a                        & 14 & 2^b ,\,3^j ,\,4^b ,\,6^a ,\,9^j ,\,13^a \\
6 & 2^a ,\,5^a                        & 15 & 2^a ,\,3^b ,\,4^a ,\,7^b ,\,11^h ,\,14^a  \\
7 & 3^b ,\,6^a                        & 16 & 3^a ,\,7^a ,\,15^a  \\
8 & 2^b ,\,3^a ,\,5^h ,\,7^a          & 17 & 2^b ,\,4^i ,\,5^c ,\,8^b ,\,10^i ,\,11^e ,\,16^a \\
9 & 2^a ,\,4^b ,\,8^a                 & 18 & 2^a ,\,5^a ,\,7^f ,\,8^a ,\,17^a \\
10 & 4^a ,\,9^a                       & 19 & 3^c ,\,4^d ,\,6^i ,\,7^b ,\,9^b,\,14^m ,\,18^a \\
11 & 2^b ,\,3^c ,\,5^b ,\,7^i ,\,10^a & 20 & 2^b ,\,3^a ,\,4^a ,\,6^b ,\,9^a ,\,13^j ,\,19^a \\
\hline
\end{array}
$$
\caption{ $(g,k)$ values for which reverse multiples exist for $g \le 20$; letters specify the Young graph. See A222817, A222819 in \cite{OEIS} for $g \le 100$.}
\label{gkTable}
\end{table}
\renewcommand{\arraystretch}{1.0}

For $g \le 20$, up to isomorphism,
only ten different Young graphs appear,
indicated in Table~\ref{gkTable} by superscripts
$a$, $b$, $c, \ldots$.
The meaning of these letters is given below.
In this list we do not count the starting node
or the edges connected to it when giving the numbers of  
nodes and edges.
All the graphs in the table are palindromic (see \S\ref{Sec1}).
If the graph is the 1089 graph or a complete or cyclic graph, the
exact form of the reverse multiples has already been discussed.

$a$. The 1089 graph (\S\ref{Sec34}).

$b, c, d$. The complete graphs $K_2$, $K_3$, $K_4$ (\S\ref{Sec35}).

$e, f$. The cyclic graphs $Z_3$, $Z_5$ (\S\ref{Sec36}).

$h$. The graph shown in Fig. \ref{Fig8_5}, with
eight nodes and 16 edges. 
The generating function is 
\beql{GF113.2}
\sC (x) = \frac{x^4(1+x)}{1-2x^2} 
       = x^4 + x^5 + 2x^6 + 2x^7 + 4x^8 + 4x^9 + 8x^{10} + 8x^{11} + \cdots. 
\eeq
The reverse multiples are the numbers $\gamma \beta$,
where $\gamma  = (1,0,1,5)_8$ if $(g,k)=(8,5)$,
$\gamma = (1,0,2,11)_{15}$ if $(g,k)=(15,11)$,
and the base-$g$ expansion of $\beta$ is palindromic and contains only 0's and 1's.

$i$. The graph shown in Fig. \ref{Fig11_7}, with
four nodes and eight edges, and
\beql{GF110.9}
\sC (x) = \frac{x^3(1+x)}{1-2x^2}
        = x^3 + x^4 + 2x^5 + 2x^6 + 4x^7 + 4x^8 + 8x^{9} + 8x^{10} + \cdots. 
\eeq
This is the first time that the smallest $(g,k)$-reverse
multiple has an odd number of digits (e.g., the 
$(11,7)$-reverse multiple $(1,1,8)_{11}$).
At first this is rather worrying,
since the Sutcliffe-Kaczynski theorem (\cite{Sut66}, \cite{Kacz68}, \cite{Pud07})
states that if there is a 3-digit reverse
multiple then there is also a two-digit reverse 
multiple.\footnote{The assertion in \cite{Pud07} that Kaczynski
shows that ``if there exists a 3-digit solution \ldots,
then deleting the middle digit gives a 2-digit solution''
is based on a mis-reading of \cite{Kacz68}.}
The explanation is that their theorem allows one to change
the multiplier $k$. Here there is a three-digit $(11,7)$-reverse multiple.
There are no two-digit $(11,7)$-reverse multiples,
but there {\em is} a two-digit $(11,3)$-reverse multiple,
$(1,4)_{11}$.
For this graph the reverse multiples are the numbers $\gamma \beta$,
where the base-$g$ expansion of $\beta$ is again a binary palindrome,
and $\gamma$ is a 3-digit number ($(1,1,8)_{11}$ if $(g,k)= (11,7)$,
$(2,5,9)_{17}$ if $(g,k)=(17,4)$, \ldots).

\begin{figure}[htb]
\begin{minipage}[b]{0.45\linewidth}
\centering
\includegraphics[width=\textwidth]{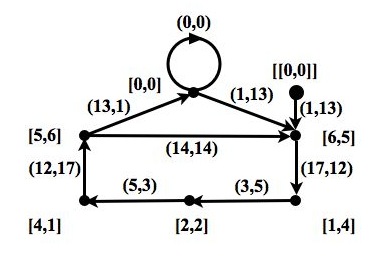}
\caption{The $(18,7)$ Young graph $Z_5$.}
\label{Fig18_7}
\vspace*{+.15in}
\end{minipage}
\hspace{0.5cm}
\begin{minipage}[b]{0.45\linewidth}
\centering
\includegraphics[width=1.0\textwidth]{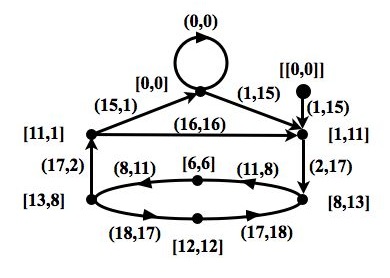}
\caption{The $(19,14)$ Young graph.}
\label{Fig19_14}
\end{minipage}
\end{figure}

$j$. The graph shown in Fig. \ref{Fig14_3}, with
six nodes and ten edges, and
\beql{GF108.8}
\sC (x) = \frac{x^5(1+x)}{1-x^2-x^4}        
= x^{5} + x^{6} + x^{7} + x^{8} + 2x^{9} + 2x^{10} + 3x^{11} + 3x^{12} + 5x^{13} + 5x^{14} + 8x^{15} +  \cdots.
\eeq
The $(g,k)$-reverse multiples are the numbers $\gamma \beta$,
where $\gamma$ is a 4-digit number
(e.g., $(1,8,12,5)_{14}$ if $(g,k)=(14,3)$) and
the base-$g$ expansion of $\beta$ is
palindromic, binary,
and does not contain any single $0$'s or $1$'s.
As in the case of the 1089 graph, $\gamma$ is not itself a reverse multiple. 

$m$. The graph shown in Fig. \ref{Fig19_14}, with
seven nodes and ten edges;
\beql{GF12.3}
\sC (x) = \frac{x^6(1+x)(1-x+x^4)}{1-x^2-x^{8}}        
= x^{6} + x^{10} + x^{11} + x^{12} + x^{13} + 2x^{14} + x^{15} + 2x^{16} + x^{17} +  \cdots ,
\eeq
where the coefficients satisfy $c_t = c_{t-2}+c_{t-8}$  for $n > 11$ (A226517).
This is $x$ times the generating function for $Z_4$.
The reverse multiples are the numbers $\gamma \beta$, where
$\beta$ is palindromic, binary,
and does not contain any run of 0's of length less than 3
or any pair of adjacent 1's. For $(g,k)= (19,14)$,
$\gamma = (1,2,11,8,17,15)_{19}$.

\vspace*{+.1in}
For $g > 20$ more complicated graphs appear.
We give three further examples.

For $(g,k) = (24,13)$, the graph $H(24,13)$ has (ignoring
the starting node)
24 nodes and 36 edges. Eight nodes disappear during the pruning process,
and the resulting 16-node 26-edge Young graph is shown in
Fig. \ref{Fig24_13}.
The pivot nodes are $[0,0]$, $[3,12]$,
$[5,5]$, $[5,12]$, $[7,0]$, $[7,7]$,
$[9,0]$, and $[12,12]$.
The generating function is surprisingly simple:
\beql{GF24.13}
\sC (x) = \frac{x^9(1+x)}{1-x^2-x^{6}}
        = x^{9} + x^{10} + x^{11} + x^{12} + x^{13} + x^{14} + 2x^{15} + 2x^{16} + 3x^{17} + 3x^{18} + 4x^{19} +  \cdots,
\eeq
and the reverse multiples are the numbers
of the form $\gamma \beta$, where
$\gamma = (23,9,8,0,16,13)_{24}$
and $\beta$ is any positive number whose base-24 expansion is
palindromic, contains only 0's and 1's,
and in which any run of 0's and 1's has length at least 3.
The smallest $(24,13)$-reverse multiple is the nine-digit
number $(1,0,9,16,18,1,6,5,13)_{24} = \gamma  (1,1,1)_{24}$,
which we can read off the path from the starting node
to the closest pivot, $[9,0]$.
The coefficients in \eqn{GF24.13} are essentially
the fourteenth-century Narayana cows sequence
(A000930) repeated, just as the coefficients
in \eqn{GF109} and \eqn{GF108.8} are the thirteenth-century
Fibonacci rabbits sequence (A000045) repeated.

\begin{figure}[htb]
\centerline{\includegraphics[width=3.0in]{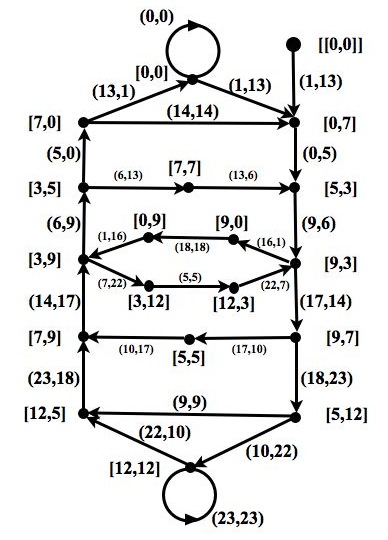}}
\caption{The $(24,13)$ Young graph.}
\label{Fig24_13}
\end{figure}

The second example is the $(24,17)$ Young graph,
which after pruning has 26 nodes and 34 edges, and is
shown in Fig. \ref{Fig24_17}.
Only the node-labels are shown, to avoid making
the diagram too cluttered. The edge-labels can be obtained from 
\eqn{Eq93.5}.
The generating function is 
\beql{GF24.17}
\sC(x) ~=~ \frac{x^{12} (1+x)}{1-x^2-x^{10}-x^{14}-x^{16}}.
\eeq
This is the smallest example of a Young graph that is
not palindromic: there is no value of $\gamma$
such that the reverse multiples divided by $\gamma$ are 
all palindromic in base 24 (one has only to check
the divisors $\gamma$ of the greatest common divisor of the first
few reverse multiples, and none of them work).

\begin{figure}[htb]
\centerline{\includegraphics[width=5.5in]{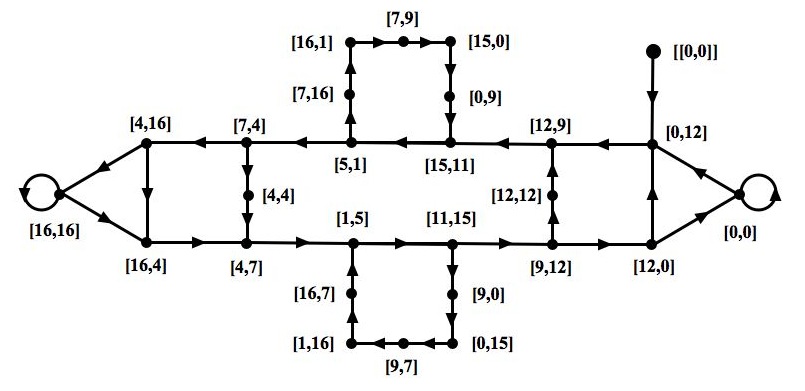}}
\caption{The $(24,17)$ Young graph.}
\label{Fig24_17}
\end{figure}

The third example is the $(40,13)$ Young graph,
with 15 nodes and 22 edges,
shown in Fig. \ref{Fig40_13}, which we mention 
because it is the largest example discussed by
Young \cite[p.~174]{Young2}.
Again only the node-labels are shown.
Figure \ref{Fig40_13} shows the structure
of these $(40,13)$-reverse multiples far more clearly
than Young's tree.
The generating function is 
\beql{GF40.13}
\sC(x) ~=~ \frac{x^5 (1+x)(1-x^2+x^4+2x^6)}{1-x^2-x^8-2x^{10}}.
\eeq
This example is also not palindromic.

\begin{figure}[htb]
\centerline{\includegraphics[width=4.5in]{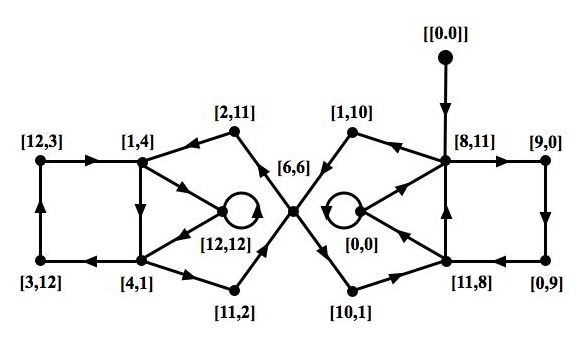}}
\caption{The $(40,13)$ Young graph.}
\label{Fig40_13}
\end{figure}

\section{Open Questions}\label{Sec5}

This extension of Young's work \cite{Young1}, \cite{Young2}
has helped clarify the properties of reverse multiples,
but it also raises a number of questions.
Besides the conjectures in Section \ref{Sec3}, we
mention three other open problems.

(i) Does the underlying directed graph characterize the Young graph?
That is, are there examples of Young graphs which are
isomorphic as directed graphs but not as Young graphs (i.e.,
have non-isomorphic sets of pivot nodes).

To make this precise, let $\sG$ be a Young graph
with underlying unlabeled directed graph $G$.
Let $\Phi$ be the map that acts on 
the labeled nodes and edges of $\sG$ by
sending $[r,s]$ to $[s,r]$, and the edge
$[r,s] \xrightarrow{(a,b)} [t,u]$
to $[u,t] \xrightarrow{(b,a)} [s,r]$.
It follows from (P3) that $\Phi$ is an involution on $\sG$,
and from Theorem \ref{ThC} that the complete graphs
are the only Young graphs that are fixed by $\Phi$.
The map $\Phi$ induces an involution of the directed graph $G$.
The question is, could there be two non-isomorphic Young graphs 
where the underlying directed graphs $G$ are isomorphic,
but where the two maps $\Phi$ act in different ways on $G$?
This seems quite possible, although no examples are presently 
known.

(ii) Which directed graphs can occur as
(the underlying directed graphs of) Young graphs?
Which ones are ``palindromic''?

(iii) Given the base $g$, which values of $k$ can occur
(cf. Table \ref{gkTable})?

\section{Acknowledgments}
Thanks to Gregory Rosenthal and Selma Rosenthal
for asking questions which led to this paper,
and to the referee for many helpful comments.


\medskip

\noindent MSC2010: 11A63

\end{document}